\documentclass[12pt,article]{memoir}
\usepackage[english.US]{babel}
\usepackage{amsthm}
\usepackage{amssymb}
\usepackage[scr]{rsfso}
\usepackage{upgreek}
\usepackage[leqno]{mathtools}
\usepackage{enumitem}
\usepackage{tikz}

\usepackage{xr-hyper}
\usepackage{hyperref}

\usepackage{url}

\usepackage{todonotes}

\definecolor{refkey}{rgb}{0,0,1}
\definecolor{labelkey}{rgb}{1,0,0}

\numberwithin{equation}{chapter}

\setsecheadstyle{\large\bfseries\raggedright}
\setsubsecheadstyle{\normalsize\bfseries\raggedright}

\theoremstyle{plain}
\newtheorem{theorem}{Theorem}
\newtheorem{proposition}[theorem]{Proposition}

\newtheorem{corollary}[theorem]{Corollary}

\theoremstyle{definition}

\newtheorem{remark}[theorem]{Remark}

\theoremstyle{plain}
\numberwithin{theorem}{chapter}

\theoremstyle{definition}
\newtoks{\thehRemark}
\newtheorem*{Remark}{\the\thehRemark}

\newenvironment{claim}[1][{\textup{(\theequation)}}]{\refstepcounter{equation}\vglue10pt
\begin{trivlist}
\item[{\hskip\labelsep#1}]}{\vglue10pt\end{trivlist}}

\newcommand{\sL}{\mathscr{L}}

\newcommand{\cX}{\mathcal{X}}

\newcommand{\sign}{\operatorname{sign}}
\newcommand{\Tr}{\operatorname{Tr}}

\newcommand{\mes}{\operatorname{mes}}

\newcommand{\supp}{\operatorname{supp}}

\newcommand{\TF}{\mathsf{TF}}

\newcommand{\y}{\mathsf{y}}

\newcommand{\bR}{\mathbb{R}}
\newcommand{\bC}{\mathbb{C}}

\newcommand{\sC}{\mathscr{C}}
\newcommand{\sH}{\mathscr{H}}

\newcommand{\fH}{\mathfrak{H}}

\begin{document}

\title{Thomas-Fermi approximation to electronic density
}

\author{Victor Ivrii\thanks{This research was supported in part by National Science and Engineering  Research Council (Canada) Discovery Grant  RGPIN 13827}}
\maketitle

\begin{abstract}
In heavy atoms  and molecules, on the distances $a \gg Z^{-1}$ from all of the nuclei (with a charge $Z_m$) we prove that $\rho_\Psi (x)$ is approximated in $\sL^p$-norm, by the Thomas-Fermi  density.
\end{abstract}

\chapter{Introduction}
\label{sect-1}
The purpose of this paper is to show that on the distances $\gg Z^{-1}$  from the nuclei   Thomas-Fermi density $\rho^\TF$ provides rather good approximation for electronic density $\rho_\Psi$ (with an error estimate in $\sL^p$-norm) 
while in \cite{ivrii:strong-scott} we considered rather short distances and used Strong-Scott density.

Let us consider the following operator (quantum Hamiltonian)
\begin{gather}
\mathsf{H}=\mathsf{H}_N\coloneqq   \sum_{1\le j\le N} H _{V,x_j}+\sum_{1\le j<k\le N}|x_j-x_k| ^{-1}
\label{eqn-1.1}\\
\shortintertext{on}
\fH= \bigwedge_{1\le n\le N} \sH, \qquad \sH=\sL^2 (\bR^3, \bC^q) \label{eqn-1.2}\\
\shortintertext{with}
H_V =-\Delta -V(x) 
\label{eqn-1.3}
\end{gather}
describing $N$ same type particles in (electrons) the external field with the scalar potential $-V$ (it is more convenient but contradicts notations of the previous chapters), and repulsing one another according to the Coulomb law.

Here $x_j\in \bR ^3$ and $(x_1,\ldots ,x_N)\in\bR ^{3N}$, potential $V(x)$ is assumed to be real-valued. Except when specifically mentioned we assume that
\begin{equation}
V(x)=\sum_{1\le m\le M} \frac{Z_m }{|x-\y_m|}
\label{eqn-1.4}
\end{equation}
where $Z_m>0$ and $\y_m$ are charges and locations of nuclei.

Mass is equal to $\frac{1}{2}$ and the Plank constant and a charge are equal to $1$ here. We assume that 
\begin{equation}
N\asymp Z=Z_1+\ldots+Z_M, \qquad Z_m\asymp Z \qquad \forall m.
\label{eqn-1.5}
\end{equation}

Our purpose is to prove that at the distances $a\gg Z^{-1}$ from the nuclei the \emph{electronic density}
\begin{equation}
\rho_\Psi (x) =N\int |\Psi (x,x_2,\ldots,x_N)|^2\,dx_2\cdots dx_N
\label{eqn-1.6}
\end{equation}
is approximated is approximated in $\sL^1(B(\y_m,a))$-norm with the relative error by the Thomas-Fermi density.

\begin{theorem}\label{thm-1.1}
\begin{gather}
\supp(U)  \subset \{x\colon a\le \ell(x)\le 2a\}, \qquad \epsilon Z^{-1}\le a\le \epsilon,\qquad |U|\le  1\,,
\label{eqn-1.7}
\intertext{where here and below}
\ell(x)= \min_{1\le m\le M} |x-\y_m|
\label{eqn-1.8}\\
\intertext{is the distance from $x$ to the nearest nucleus,}
\zeta = \left\{\begin{aligned}
&Z^{1/2}\ell^{-1/2} && \ell \le Z^{-1/3},\\
&\ell^{-2} &&\ell \ge Z^{-1/3}, 
\end{aligned}\right.
\label{eqn-1.9}
\end{gather}
 Then  under assumption
\begin{equation}
\min_{1\le m< m' \le M} |\y_{m}-\y_{m'}|\ge Z^{-1/3+\sigma}
\label{eqn-1.10}
\end{equation}
the following estimate holds:
\begin{gather}
|\int U(x) \bigl(\rho_\Psi (x)- \rho^\TF (x)\bigr)\,dx| \le  C\Bigl(E\|U \|_{\sL^1} + F^{1/3} \|U\|_{\sL^1}^{2/3}  + G\Bigr)
\label{eqn-1.11}\\
\shortintertext{with}
E= \zeta^2 a^{-1},\qquad  F=\zeta^{4/3}Z^{5/3-\delta},\qquad G=\zeta^{-2}Z^{5/3-\delta}
\label{eqn-1.12}
\end{gather}
and  $\delta=\delta(\sigma)$, $\delta>0$ if $\sigma>0$ and $\delta=0$ if $\sigma=0$.
\end{theorem}

\begin{corollary}\label{cor-1.2}
Let \underline{either} $a\asymp Z^{-1}$ and $X\subset B(\y_m,a)$ \underline{or} $Z^ {-1}\le a \le 1$ and~$X\subset B(\y_m,a)\setminus B(\y_m,a/2)$. Then
\begin{enumerate}[label=(\roman*), wide, labelindent=0pt]
\item\label{cor-1.2-i}
The following estimate holds:
\begin{gather}
\|\rho_\Psi -\rho^\TF \|_{\sL^1 (X)} \le C\Bigl(E\mes(X) + F(\mes(X))^{2/3}+G\Bigr).
 \label{eqn-1.13}
\end{gather}

\item\label{cor-1.2-ii}
Assume that $|\rho_\Psi |\le \omega $ in $X$ with $\omega  \ge Z^{3/2}a^{-3/2}$\,\footnote{\label{foot-1} Such estimates could be found in \cite{ivrii:el-den}: $\omega  = Z^3$ for $a\le Z^{-8/9}$, $\omega =Z^{19/9} a^{-1}$ for $Z^{-8/9}\le a\le Z^{-7/9}$, $\omega =Z^{197/90}a^{-9/10}$ for $Z^{-7/9}\le a\le Z^{-1/3}$, $\omega=Z^{17/9}a^{-9/5}$ for $Z^{-1/3}\le a\le Z^{-5/18}$ and $\omega=Z^{19/9}a^{-1}$ for $a\ge Z^{-5/18}$.}. Then for $p=2,3,\ldots$ the following estimate holds:
\begin{gather}
  \| \rho_\Psi -\rho^\TF  \|_{\sL^p (X)}\le C\omega^{1-1/p}J_p^{1/p} 
  \label{eqn-1.14}\\
  \intertext{where $J_p$ with $p=1,2,\ldots$ is a solution to a recurrent relation}
J_p= \omega^{-1}E  J_{p-1} +( \omega^{-2}F)^{1/3}  J_{p-1}^{2/3}+ G, \qquad J_0= \omega\mes(X) .
\label{eqn-1.15}
\end{gather}
\end{enumerate}
\end{corollary}

\begin{remark}\label{rem-1.3}
\begin{enumerate}[label=(\roman*), wide, labelindent=0pt]
\item\label{rem-1.3-i} 
As we mentioned, in \cite{ivrii:strong-scott} $\|\rho_\Psi-\rho_m\|_{\sL^p}$ is estimated as $a=|x-\y_m|\le Z^{-1/2-\varkappa}$ with arbitrarily small exponent $\varkappa>0$, where $\rho_m$ is  the electronic density for a single atom in the model with no interactions between electrons. This estoimate is obtained almost exclusively by functional-analytic methods, while in the current paper we just start from the estimate (\ref{eqn-2.1}) and apply methods of non-smooth microlocal analysis.

As $\|U\|_{\sL^(X)}\asymp a^{3}$ in Theorem~\ref{thm-1.1} and  $\mes(X)\asymp a^3$ in Corollary~\ref{cor-1.2} the estimates of that paper are better than the corresponding estimates of this one, but for relatively small $\|U\|_{\sL^1(X)}$  and $\mes(X)$ those of the present paper may be better.

However, as $\|U\|_{\sL^1(X)}$  and $\mes(X)$  are too small the right-hand expressions of (\ref{eqn-1.11}), (\ref{eqn-1.13}) and (\ref{eqn-1.14}) may be larger, than $\zeta^3 \int _X |U|\,dx $ or $\zeta^3\mes(X)$ respectively in which case those are no more remainder estimates.

\item\label{rem-1.3-ii} 
The methods developed in this paper work without any modifications in the case of the self-generated magnetic field (see Chapter~\ref{monsterbook-sect-27} of \cite{monsterbook}), in the relativistic settings (see~\cite{ivrii:relativistic-1} and in the relativistic settings and self-generated magnetic field combined (see~\cite{ivrii:relativistic-2}). In all these cases exactly the same remainder estimate (\ref{eqn-1.9}) could be derived. In the relativistic case the relative error (when we use semiclassical non-relativistic formula instead of relativistic one ) would be of magnitude $1/(Za)$ which is less than the relative semiclassical error; the same is true when we introduce self-generated magnetic field, except we know the upper bound rather than the magnitude of the error. 

On the shorter distances (like in \cite{ivrii:strong-scott}), however, the asymptotics of $\rho_\Psi$ would be different in the relativistic case for sure because the Scott correction term is different.
\end{enumerate}
\end{remark}

The proof is short. It is based on estimate (\ref{eqn-2.1}) which is based on functional-analytic arguments combined with the sharp asymptotics for the ground state energy (which is also used in \cite{ivrii:strong-scott}) and on microlocal analysis of  a one-particle Hamiltonian with a potential $V=W^\TF+\varsigma U$ where $U$ is supported in $\cX$ and satisfies $|U(x)|\le  W^\TF$. \todo[color=green]{WORK}

\chapter{Proof of main results}
\label{sect-2}

The following estimate  is  (\ref{strongscott-eqn-4.1}) of \cite{ivrii:strong-scott}, proven under assumptions (\ref{eqn-1.8}):
\begin{equation}
\pm \varsigma \int U\rho_\Psi \,dx  \le 
\Tr [(H_{W+\nu}^-]   -  \Tr [H_{W\pm \varsigma U+\nu}^-] +C Z^{5/3-\delta}
\label{eqn-2.1}
\end{equation}
with $\delta=\delta(\sigma)>0$ as $\sigma>0$ and $\delta=0$ as $\sigma=0$, where $W=W^\TF$ is Thomas-Fermi potential. 

We assume that
\begin{gather}
\supp(U)  \subset \{x\colon a\le \ell(x)\le 2a\}, \qquad \epsilon Z^{-1}\le a\le \epsilon,\qquad |U|\le  \zeta^2\,,
\label{eqn-2.2}
\intertext{where here and below}
\zeta = \left\{\begin{aligned}
&Z^{1/2}\ell^{-1/2} && \ell \le Z^{-1/3},\\
&\ell^{-2} &&\ell \ge Z^{-1/3}, 
\end{aligned}\right.
\label{eqn-2.3}
\shortintertext{and}
0<\varsigma \le \epsilon.
\label{eqn-2.4}
\end{gather}

We rewrite the right-hand expression as
\begin{equation}
-\int_0^\varsigma \Tr \bigl(U \uptheta (-H_{W\pm s U+\nu})\bigr)\,ds.
\label{eqn-2.5}
\end{equation}
After rescaling $x\mapsto x a^{-1}$, $\tau\mapsto \tau \zeta^{-2}$ we are in the semiclassical settings with 
$h= 1/(\zeta a)$. Condition to $a$ ensures that $h\le 1$. 

Due to Tauberian method we can rewrite the integrand  as
\begin{gather}
h^{-1}\int_{-\infty}^0 F_{t\to h^{-1}\tau} \bigl(\bar{\chi}_T (t) \Gamma (U(x) u_s)\bigr)\,d\tau
\label{eqn-2.6}\\
\intertext{with an error, essentially not exceeding}
C\varsigma T^{-1} \sup _{|\tau|\le \epsilon_1} |F_{t\to \tau} \bigl(\bar{\chi}_T (t) \Gamma (|U(x)| u_s )\bigr)|
\label{eqn-2.7}
\end{gather}
(we will write an exact statement later) where $\bar{\chi}_T(t)=\bar{\chi}(t/T)$, $\bar{\chi}\in \sC_0^\infty ([-1,1])$, equal $1$ on $[-\frac{1}{2},\frac{1}{2}]$, $h\le T\le \epsilon$,  and
\begin{claim}\label{eqn-2.8}
$u_s(x,y,t) $ is a Schwartz kernel of $\exp( -i h^{-1}t H_{W+s U+\nu})$,
\end{claim}
$\Gamma  v= \int (\Gamma_x v)\, dx$ and $\Gamma_x v=v(x,x,\cdot)$.

Under assumptions (\ref{eqn-2.2})--(\ref{eqn-2.5}), using two-terms successive approximation method with unperturbed operator $H_{W+\nu}$ and perturbation $s U$, we can evaluate  
\begin{gather}
F_{t\to h^{-1}\tau} \bigl(\bar{\chi}_T (t) \Gamma (U(x) u_s)\bigr)
\label{eqn-2.9}\\
\shortintertext{as}
F_{t\to h^{-1}\tau} \bigl(\bar{\chi}_T (t) \Gamma (U(x) u_0)\bigr)
\label{eqn-2.10}
\end{gather}
with an error, not exceeding  
\begin{gather}
C \|U\|_{\sL^1} a^{-3}  h^{-1}T^2 \varsigma   \times h^{-3}. 
\label{eqn-2.11}\\
\intertext{Indeed, one can see easily that the absolute value of $n$-th term is estimated by}
Ch^{1-n}T^{n} \varsigma^{n-1} \times \|U\|_{\sL^1} a^{-3} h^{-3}
\label{eqn-2.12}
\end{gather}
for $n\ge 1$ and (\ref{eqn-2.11}) is (\ref{eqn-2.12}) with $n=1$.

However, we can do better than this:

\begin{proposition}\label{prop-2.1}
 Under assumptions \textup{(\ref{eqn-2.2})}--\textup{(\ref{eqn-2.4})} and $T\le \epsilon$ the following estimate holds:
\begin{gather}
|F_{t\to h^{-1}\tau} \bigl(\bar{\chi}_T (t) \Gamma (U(x) u_0)\bigr) |\le  C \|U\|_{\sL^1}  a^{-3} h^{-2}\,.
\label{eqn-2.13}
\end{gather}
\end{proposition}

\begin{proof}
Estimate (\ref{eqn-2.13}) is due to  semiclassical microlocal arguments applied to $\Gamma_x u$, in which case non-smoothness of $U$ plays no role.  

Indeed,   for $\ell(x) \le \epsilon Z^{-1/3}$ we have  $W\coloneqq W^\TF\asymp Z\ell^{-1}$, and  $0\le -\nu \le C Z^{4/3}$, so in this case  operator $H_{W+\nu}$ is $x$-microhyperbolic on the energy level $0$. 

 Furthermore,  this operator  is also $x$-microhyperbolic on the energy level $0$ for 
$\ell (x) \le \epsilon_0 |\nu|^{1/4}\asymp \epsilon_0 (Z-N)_+^{-1/3}$\,\footnote{\label{foot-2} Recall that $|\nu|\asymp (Z-N)_+^{1/4}$.} and elliptic for $\ell (x) \ge C_0 |\nu|^{1/4}$. 

Finally, as $\epsilon_0 |\nu |^{-1/4} \le \ell(x) \le C_0 |\nu|^{-1/4}$ $x$-microhyperbolicity fails but estimate (\ref{eqn-2.13})   follows from the partition-rescaling technique of Subsection~\ref{monsterbook-sect-5-1-1} of \cite{monsterbook} (since dimension is $3$).  We leave easy details to the reader.
\end{proof}

Thus, in the framework of Proposition~\ref{prop-2.1} we can estimate expression (\ref{eqn-2.9}):  
\begin{equation}
|F_{t\to h^{-1}\tau} \bigl(\bar{\chi}_T (t) \Gamma (U(x) u_s)\bigr)|\le C\|U\|_{\sL^1} a^{-3}\bigl(h^{-2}  +   \varsigma  T^2 h^{-4}\bigr)\,,
\label{eqn-2.14}
\end{equation}
where the first term in the right-hand expression is due to (\ref{eqn-2.13}) and the second one is (\ref{eqn-2.11}). 

Using this estimate and standard Tauberian arguments we can estimate 
\begin{multline}
|\Tr \bigl[ U\uptheta (- H_{W+sU+\nu})\bigr] \\
\shoveright{-h^{-1}\int_{-\infty}^0  F_{t\to h^{-1}\tau}\bigl( \bar{\chi}_T(t)\Gamma U(x)u_s(.,.,t)\bigr)\,d\tau |}\\
\le C\|U\|_{\sL^1} a^{-3}\bigl(T^{-1}h^{-2}  +   \varsigma  T h^{-4}\bigr)\,,
\label{eqn-2.15}
\end{multline} 
where the right-hand expression is the right-hand expression of (\ref{eqn-2.14}), multiplied by $T^{-1}$.

On the other hand, without these microlocal arguments and for $\varphi \in \sC_0^\infty ([-1,-\frac{1}{2}])$, 
$L T\ge h$ the following estimate holds:
\begin{multline}
|\int_{-\infty}^0  F_{t\to h^{-1}\tau} \varphi _L(\tau) \bigl(\bar{\chi}_T (t) \Gamma (U(x) (u_0-u_s)\bigr)\,d\tau|\le\\
C_kL \|U\|_{\sL^1}a^{-3}  h^{-1}T^2 \varsigma h^{-3}\times (h/TL)^k L 
\label{eqn-2.16}
\end{multline}
with arbitrarily large exponent $k$. \enlargethispage{\baselineskip}

Therefore, if we replace in the left-hand expression $\varphi _L(\tau)$ by $1$,  the result would not exceed the same right-hand expression with 
$L=h/T$, i.e.  $C_k \|U\|_{\sL^1} a^{-3}T\varsigma h^{-3}$, and therefore 
\begin{multline}
h^{-1}|\int_{-\infty}^0 F_{t\to h^{-1}\tau} \bigl(\bar{\chi}_T (t) \Gamma (U(x) (u_s-u_0))\bigr)\,d\tau|\\
\le C\|U\|_{\sL^1} a^{-3} \varsigma  T h^{-4}.
\label{eqn-2.17}
\end{multline}
Therefore  
\begin{claim}\label{eqn-2.18}
Estimate (\ref{eqn-2.15}) remains true, if we replace there $u_s$ by $u_0$.
\end{claim}

However, then we can replace Tauberian expression by the corresponding Weyl expression 
\begin{equation}
\frac{q}{6\pi^2} \int U (W+\nu)_+^{3/2} \,dx = \int U(x)\rho^\TF(x) \,dx\
\label{eqn-2.19}
\end{equation}
because $W=W^\TF$.

Plugging it into (\ref{eqn-2.15}) and minimizing  by $T\in [h,1]$ we get 
\begin{multline}
|\Tr \bigl[ U\bigl(\uptheta (- H_{W+sU+\nu})\bigr] -\int U(x)\rho^\TF (x)\,dx |\\ 
\le C \|U\|_{\sL^1}  a^{-3}\bigl(\varsigma^{1/2}h^{-3} + h^{-2}\bigr)
\label{eqn-2.20}
\end{multline} 
for $T\asymp \min\bigl( h\varsigma^{-1/2},\,1\bigr)$.

Since
\begin{gather}
\Tr \bigl[H_{W+\nu}^-\bigr]   -  \Tr \bigl[H_{W\pm \varsigma U+\nu}^-\bigr] = \int _0^\varsigma  \Tr \bigl[ U\bigl(\uptheta (- H_{W+sU+\nu})\bigr] \,ds
\label{eqn-2.21}
\end{gather}
inequality (\ref{eqn-2.1}) implies, after we divide by $\varsigma \zeta^2$,
\begin{multline}
|\int \bar{U} (x)\bigl(\rho_\Psi(x)-\rho^\TF(x)\bigr)\,dx|\\
 \le  C \|\bar{U}\|_{\sL^1}  a^{-3}\bigl(\varsigma^{1/2}\zeta^3 a^3 +  \zeta^2 a^2\bigr) + \varsigma^{-1}\zeta^{-2}Z^{5/3-\delta}.
 \label{eqn-2.22}
 \end{multline}
with $\bar{U}=\zeta^{-2}U$, satisfying (\ref{eqn-1.7}). \emph{From now on we redefine $U\coloneqq \bar{U}$.}

Minimizing the right-hand expression by $\varsigma \le \epsilon$ we arrive to estimate (\ref{eqn-1.11})--(\ref{eqn-1.12}).
Theorem~\ref{thm-1.1} has been proven.

Next, picking up 
\begin{equation*}
U(x)=\sign (\rho_\Psi (x)- \rho^\TF (x))\chi_X(x)
\end{equation*}
with $\chi_X$ characteristic function of $X$ we arrive to (\ref{eqn-1.11}) with the left-hand expression replaced by
$\|\rho_\Psi - \rho^\TF \|_{\sL^1(X)}$ and the right-hand expression with $\|U\|_{\sL^1}= \mes(X)$, which proves Corollary~\ref{cor-1.2}\ref{cor-1.2-i}.

Furthermore, in the framework of Corollary~\ref{cor-1.2}\ref{cor-1.2-ii} we pick up
\begin{gather*}
U(x)=\omega^{1-p} \sign (\rho_\Psi (x)- \rho^\TF (x))|(\rho_\Psi (x)- \rho^\TF (x))|^{p-1} \chi_X(x)
\end{gather*}
and we conclude that  $J^*_p\coloneqq\omega^{1-p}\| \rho_\Psi - \rho^\TF\|^p _{\sL^p(X)}$ satisfies 
\begin{gather}
J^*_p\le C\bigl( \omega^{-1} E  J^*_{p-1} + (\omega^{-2} F) ^{1/3} J_{p-1}^{*\,2/3}+ G \bigr), \qquad J^*_0\le \omega\mes(X) 
\label{eqn-2.24}
\end{gather}
with $E$, $F$ and $G$ defined by (\ref{eqn-1.12}) and therefore $J^*_p\le C_p J_p$ with $J_p$ defined as solutions of 
(\ref{eqn-1.15}).  This conclude the proof of Corollary~\ref{cor-1.2}\ref{cor-1.2-ii}.

\end{document}